\documentclass[11pt,reqno]{amsart}
\usepackage{graphicx}
\usepackage{float}
\usepackage[caption = false]{subfig}
\usepackage{enumerate}
\usepackage{mathtools}
\usepackage{breqn}
\usepackage{array, geometry, graphicx}
\usepackage{amsmath,amsfonts,paralist,amssymb,amsthm,mathrsfs}
\usepackage{pdfpages}
\usepackage{amsfonts}
\usepackage{amsthm}
\usepackage{array}
\usepackage{mathtools}

\newtheorem{theorem}{Theorem}[section]
\newtheorem{corollary}[theorem]{Corollary}
\newtheorem{lemma}[theorem]{Lemma}

\theoremstyle{definition}
\newtheorem{definition}[theorem]{Definition}
\theoremstyle{remark}
\newtheorem{remark}[theorem]{Remark}
\numberwithin{equation}{section}
\DeclareMathOperator{\RE}{Re} \DeclareMathOperator{\IM}{Im}\DeclareMathOperator{\asin}{sin^{-1}}\DeclareMathOperator{\atan}{tan^{-1}}

\begin{document}
	\title{On a Class of certain Non-Univalent Functions}
	\thanks{The second author is supported by The Council of Scientific and Industrial Research(CSIR). Ref.No.:08/133(0030)/2019-EMR-I}
	\author{S. Sivaprasad Kumar}
	\address{Department of Applied Mathematics, Delhi Technological University, Delhi--110042, India}
	\email{spkumar@dce.ac.in}
	\author[Pooja Yadav]{Pooja Yadav}
	\address{Department of Applied Mathematics, Delhi Technological University, Delhi--110042, India}
	\email{poojayv100@gmail.com}

	
	\begin{abstract}
In this paper, we introduce a family of analytic functions given by
$$\psi_{A,B}(z):= \dfrac{1}{A-B}\log{\dfrac{1+Az}{1+Bz}},$$
which maps univalently the unit disk onto either elliptical or strip domains, where either $A=-B=\alpha$ or $A=\alpha e^{i\gamma}$ and $B=\alpha e^{-i\gamma}$ ($\alpha\in(0,1]$ and $\gamma\in(0,\pi/2]$). We study a class of non-univalent analytic functions defined by
 \begin{equation*}
  \mathcal{F}[A,B]:=\left\{f\in\mathcal{A}:\left( \dfrac{zf'(z)}{f(z)}-1\right)\prec\psi_{A,B}(z)\right \}.
\end{equation*}
Further, we investigate various characteristic properties of $\psi_{A,B}(z)$ as well as functions in the class $\mathcal{F}[A,B]$ and obtain  the sharp radius of starlikeness of order $\delta$ and  univalence for the functions in $\mathcal{F}[A,B]$. Also, we find the sharp radii for functions in $\mathcal{BS}(\alpha):=\{f\in\mathcal{A}:zf'(z)/f(z)-1\prec z/(1-\alpha z^2),\;\alpha\in(0,1)\}$, $\mathcal{S}_{cs}(\alpha):=\{f\in\mathcal{A}:zf'(z)/f(z)-1\prec z/((1-z)(1+\alpha z)),\;\alpha\in(0,1)\}$ and others  to be in the class $\mathcal{F}[A,B].$
	\end{abstract}
	\maketitle
	
\section{Introduction}

Let $\mathcal{A}$ denote the class of analytic functions $f$ defined on the open unit disk $\mathbb{D}=\{z\in\mathbb{C}:|z|<1\}$, such that $f(z)=z+a_2z^2+a_3z^3+\cdots.$
 The subclass of $\mathcal{A}$ consisting of all univalent functions is denoted by $\mathcal{S}$. For two analytic functions $f$ and  $g,$ we say that $f$ is subordinate$ (\prec)$ to
$g$,  if $f(z) = g(\omega(z))$ where  $\omega(z)$ is a Schwarz function. Let $\mathcal{S}^*$ be the class of starlike functions, consists of functions $f\in\mathcal{S}$ such that $zf'(z)/f(z)\prec(1+z)/(1-z).$
 In \cite{maminda}, Ma and Minda unifies all the subclasses of $\mathcal{S}^*$ as follows: \begin{equation}\label{mamin}
    \mathcal{S}^*(\Phi):=\left\{f\in\mathcal{S}:\dfrac{zf'(z)}{f(z)}\prec\Phi(z)\right\},
\end{equation} where $\Phi$ is univalent with $\Phi'(0)>0$ satisfying \begin{itemize}\item[A.] $\Phi(z)\prec(1+z)/(1-z)$ \item[B.] $\Phi(\mathbb{D})$ is symmetric about real axis and starlike with respect to $\Phi(0)=1$.\end{itemize}
Further, Kumar and Banga \cite{nonmaminda}, called such $\Phi$ as {\it Ma-Minda function} and the class of all Ma-Minda functions is denoted by $\mathcal{M}.$ In \cite{nonmaminda}, authors extensively studied Ma-Minda functions and classified them as  {\it non-Ma-Minda function of type-A} whenever condition A doesn't hold, the class of all such functions is denoted by $\widetilde{\mathcal{M}}_{A}.$    In past many authors studied the class (\ref{mamin}) for various choices of   $\Phi\in\mathcal{M}$, such as   $\mathcal{S}^*(2/(1+e^{-z}))=:\mathcal{S}^*_{SG}$ \cite{sigmoid}, $\mathcal{S}^*(e^z)=:\mathcal{S}^*_{e}$ \cite{exp}, $\mathcal{S}^*(z+\sqrt{1+z^2})$ \cite{raina}, $\mathcal{S}^*(\sqrt{1+z})=:\mathcal{S}_{L}^*$ \cite{sokol}, etc. Also   the well known classes $ \mathcal{S}^{*}(A,B)$ and $\mathcal{SS}^*(\alpha)$ are also obtained by taking $\Phi(z)(\in\mathcal{M})$ as $(1+Az)/(1+Bz)$ $(-1\leq B<A \leq1)$ and $((1+z)/(1-z))^\alpha$ $(\alpha\in(0,1])$ respectively.  In particular, $\mathcal{S}^{*}(1,-1)=:\mathcal{S}^{*}$, $\mathcal{S}^{*}(1-2\delta,-1)=:\mathcal{S}^{*}(\delta)$ ($\delta\in[0,1)$), the class of starlike functions of order $\delta$. Robertson \cite{robertson}
introduced and investigated the class of starlike functions of order $\delta$, denoted by $\mathcal{S}_{\delta}^*$ for $\delta\leq1$. Note that if $\delta<0$, then the functions in $\mathcal{S}_{\delta}^*$ may not be univalent, i.e. if $\delta<0$, then $\Phi(z)=(1+(1-2\delta)z)/(1-z)\in\widetilde{\mathcal{M}}_{A}$ and $\mathcal{S}^*_\delta\not\subset\mathcal{S}^*.$ Later many authors considered the classes associated with non-Ma-Minda functions of type-A such as the class by Uralegaddi \cite{uralegaddi}  $$\mathscr{M}(\beta)=\left\{f\in\mathcal{A}:\dfrac{zf'(z)}{f(z)}\prec\dfrac{1+(2\beta-1)z}{1-z},\;\beta>1\right\},$$   the classes  associated with the Booth Lemniscate and Cissoid of Diocles  \begin{equation*}
    \mathcal{BS}^*(\alpha):=\left\{f\in\mathcal{A}:\left(\dfrac{zf'(z)}{f(z)}-1\right)\prec\dfrac{z}{1-\alpha z^2},\;\alpha\in[0,1)\right\},
\end{equation*}  \begin{equation*}
    \mathcal{S}^*_{cs}(\alpha):=\left\{f\in\mathcal{A}:\left(\dfrac{zf'(z)}{f(z)}-1\right)\prec\dfrac{z}{(1-z)(1+\alpha z)},\;\alpha\in[0,1)\right\},
\end{equation*} which were introduced by \cite{kargar1} and \cite{masih} respectively. Similarly, Masih and Kanas \cite{masih1}  studied the class
    $\mathcal{ST}^*_{L}(s):=\{f\in\mathcal{A}:zf'(z)/f(z)$ $\prec\mathbb{L}_s(z)\}$ associated with the Lima\c{c}on of Pascal $\mathbb{L}_s(z)=(1+sz)^2$, $s\in[-1,1]\backslash\{0\}$. Note that $\mathbb{L}_s(z)\in\widetilde{\mathcal{M}}_{A}$, whenever $|s|>1/\sqrt{2}.$  Recently,  Kumar and Gangania \cite{kamal} studied the following class:
\begin{equation}\label{kamal}
    \mathcal{F}(\psi):=\left\{f\in\mathcal{A}:\dfrac{zf'(z)}{f(z)}-1\prec\psi(z)\right\},
\end{equation} where $\psi(z)\in\mathcal{S}^*$. If $1+\psi(z)\prec(1+z)/(1-z)$ and $1+\psi(\mathbb{D})$ is symmetric about the real axis with $\psi'(0)>0$ then $\mathcal{F}(\psi)$ reduces to the class $\mathcal{S}^*(1+\psi)$, a Ma-Minda class and if $1+\psi(z)\not\prec(1+z)/(1-z)$ then the functions in $ \mathcal{F}(\psi)$ may not be univalent, thus $\mathcal{F}(\psi)\not\subseteq\mathcal{S}^*$.  Note that all above mentioned classes associated with non-Ma-Minda functions of type-A are   the special cases of $\mathcal{F}(\psi)$ such as $ \mathcal{F}((1+(2\beta-1)z)/(1-z))=:\mathscr{M}(\beta),$ $\mathcal{F}(z/(1-\alpha z^2))=: \mathcal{BS}^*(\alpha)$, $\mathcal{F}(z/((1-z)(1+\alpha z)))=: \mathcal{S}^*_{cs}(\alpha)$ and $\mathcal{F}((1+sz)^2)=:\mathbb{L}_s(z).$

Motivated by the above, we define a family of functions $\psi_{A,B}(z)$ as
\begin{equation}\label{def1}
    \psi_{A,B}(z):=\dfrac{1}{A-B}\log\left(\dfrac{1+Az}{1+Bz}\right)\qquad z\in\mathbb{D},
\end{equation}where either  $A=-B=\alpha$ or $A=\alpha e^{i \gamma}$ and $B=\alpha e^{-i\gamma}$,  for $0<\alpha\leq1$ and $0<\gamma\leq\pi/2.$  The function defined by (\ref{def1}) is analytic everywhere except the branchcuts $\{-\infty<\RE Az\leq-1,\;\IM Az=0\}\cup\{-\infty<\RE Bz\leq-1,\;\IM Bz=0\}$ with $\psi_{A,B}(0)=\psi'_{A,B}(0)-1=0$, i.e. $\psi_{A,B}(z)\in\mathcal{A}.$  Now we introduce the following  classes associated with $\psi_{A,B}(z)$:
\begin{definition}
Let $p\in\mathcal{S}$. Then $p\in\mathcal{L}[A,B]$ if and only if
\begin{equation*}
    p(z)\prec \psi_{A,B}(z).
\end{equation*} \end{definition}
\begin{definition}
Let $f\in\mathcal{A}$. Then $f\in\mathcal{F}[A,B]$ if and only if
\begin{equation}\label{sq[a,b]}
    \dfrac{zf'(z)}{f(z)}-1\prec \psi_{A,B}(z).
\end{equation} \end{definition}
In section 2 we investigate various characteristic properties of functions in the classes $\mathcal{L}[A,B]$ and $\mathcal{F}[A,B].$  We also obtain the extremal function of the class $\mathcal{F}[A,B]$ which is non-univalent and study its various geometric properties. Further, in section 3 we derive the sharp radius  of starlikeness of order $\delta$, univalence  for the functions in $\mathcal{F}[A,B]$ and also find the sharp radii for functions in $\mathcal{BS}(\alpha)$, $\mathcal{S}_{cs}(\alpha)$ and others to be in the class $\mathcal{F}[A,B].$

\section{Characteristics of $\mathcal{L}[A,B]$ and $\mathcal{F}[A,B]$}
In this section, we aim to find  various geometric properties of $\psi_{A,B}(z)$ and obtain certain  bounds and inclusion relations for functions in the class  $\mathcal{L}[A,B]$. Further, we deduce the extremal function and derive the growth and covering theorems  for the class $\mathcal{F}[A,B]$.   Note that
\begin{equation*}
    \psi_{A,B}(z)=\sum_{n=1}^{\infty}(-1)^{n+1}C_{n}z^n,
\end{equation*}
where \begin{equation}\label{cn}
    C_{n}=\dfrac{A^n-B^n}{n(A-B)}=\dfrac{1}{n}\sum_{k=0}^{n-1}A^k B^{n-1-k}\qquad (n=1,2,\hdots).
\end{equation}
\begin{remark}From (\ref{cn}) we have
\begin{itemize}
    \item[(i)] Since $C_{n}$ is real valued for all $n$, $\psi_{A,B}(\mathbb{D})$ is symmetric with respect to real axis.
\item[(ii)] $|C_{n}|=\left|\dfrac{1}{n}\sum_{k=0}^{n-1}A^k B^{n-1-k}\right|\leq\alpha^{n-1}\leq1$.
\end{itemize}
\end{remark}

\begin{lemma}
The function $\psi_{A,B}(z)$ is convex and univalent on $\mathbb{D}.$
\end{lemma}
\begin{proof}
Let $H(z):=1+z\psi''_{A,B}(z)/\psi'_{A,B}(z).$ By \cite[Corollary 3]{sakaguchi}, it is enough to show that $\RE H(z)>0$.
Since $|A|=|B|=\alpha\leq1$,  using basic calculation with \cite[Theorem 1]{siva}, we obtain
\begin{align*}
    \RE H(z)&=\RE\left(-1+\dfrac{1}{1+Az}+\dfrac{1}{1+Bz}\right)\\
    &> -1+\dfrac{2}{1-\alpha}>0.
\end{align*}Thus the result holds.
\end{proof}

\begin{theorem}\label{bounds}
Let $p\in\mathcal{L}[A,B]$, then in the disk $\mathbb{D}_r=\{z\in\mathbb{C}:|z|\leq r<1\}$,
\begin{itemize}
    \item[(i)] For $A=-B=\alpha$, we have
    \begin{equation}\label{a=-b,mod}
    |\RE p(z)|\leq\dfrac{1}{2\alpha}\log\left(\dfrac{1+\alpha r}{1-\alpha r}\right)\end{equation}
    \begin{equation}\label{a=-b,arg}
    |\IM p(z)|\leq\dfrac{1}{2\alpha}\asin\left(\dfrac{2\alpha r}{1+\alpha^2 r^2}\right).
\end{equation}
\item[(ii)] For $A=\alpha e^{i \gamma}$ and $B=\alpha e^{-i\gamma}$, we have
 \begin{equation}\label{modd}
  \dfrac{-1}{2\alpha\sin{\gamma}}(\eta-\tau)\leq\RE p(z)\leq\dfrac{1}{2\alpha\sin{\gamma}}(\eta+\tau).\end{equation}
  \begin{equation}\label{argg}
 \dfrac{1}{2\alpha\sin{\gamma}}\log T_2 \leq\IM p(z)\leq\dfrac{1}{2\alpha\sin{\gamma}}\log T_1.
\end{equation} 
where $$\eta:=\asin\left(\dfrac{2\alpha r\sin\gamma}{\sqrt{1+\alpha^4r^4-2\alpha^2r^2\cos2\gamma}}\right),$$ $$\tau:=\atan\left(\dfrac{-\alpha^2r^2\sin{2\gamma}}{1-\alpha^2r^2\cos{2\gamma}}\right)$$
and $$T_j:=\left(\dfrac{\sqrt{1+\alpha^4r^4-2\alpha^2r^2\cos2\gamma}+(-1)^j2\alpha r\sin\gamma}{1-\alpha^2r^2}\right)^{-1} \;\;(j=1,2).$$
\end{itemize}
\end{theorem}

\begin{proof}
Let $w$ be the Schwarz function such that $w(0)=0$ and $|w(z)|\leq|z|=r<1$ for all $z\in\mathbb{D}$. Define the function $F_w:\mathbb{D}\to \mathbb{C}$ such that $p(z)=\psi_{A,B}(w(z))=\dfrac{1}{A-B}\log\left(F_w(z)\right)$. We have
\begin{equation*}
    F_w(z):=\dfrac{1+A w(z)}{1+B w(z)},
\end{equation*}which maps unit disk into the disk
\begin{equation*}
    F_w(\mathbb{D})\subset\left\{\zeta\in\mathbb{C};\bigg|\zeta-\dfrac{1-A\overline{B}}{1-|B|^2}\bigg|<\dfrac{|A-B|}{1-|B|^2}\right\}.
\end{equation*}As $|A|=\alpha\in(0,1],$ therefore origin $O\not\in F_w(\mathbb{D})$ (for more details see \cite{siva}). Now  using \cite[Theorem 1]{siva} with $|z|=r<1$, we have
\begin{equation*}
    \dfrac{|1-A\overline{B}r^2|-|A-B|r}{1-|B|^2r^2}\leq|F_w(z)|\leq\dfrac{|1-A\overline{B}r^2|+|A-B|r}{1-|B|^2r^2},
    \end{equation*}
    \begin{equation*}
	\bigg|\arg F_w(z)-\atan{\dfrac{\IM(A\overline{B})r^2}{\RE(A\overline{B})r^2-1}}\bigg|<\asin\dfrac{|A-B|r}{|1-A\overline{B}r^2|}.
\end{equation*}   Since $\log{f(z)}=\log|f(z)|+i\arg{f(z)}$ and $p(z)=\dfrac{1}{A-B}\log\left(F_w(z)\right)$, thus for $A=-B=\alpha$
we obtain $(\ref{a=-b,mod})$ and $(\ref{a=-b,arg})$
and for $A=\alpha e^{i \gamma}$ and $B=\alpha e^{-i\gamma}$,
 we have $(\ref{modd})$ and $(\ref{argg})$.
\end{proof}

For the sake of computational convenience, we shall make the following assumptions:\begin{gather*}
   h_1:=\dfrac{1}{2\alpha}\log\left(\dfrac{1+\alpha }{1-\alpha }\right),\;\;\;\;\;\;h_2:=\dfrac{1}{2\alpha}\asin{\left(\dfrac{2\alpha }{1+\alpha^2 }\right)},\\ k_1:=\dfrac{1}{2\alpha\sin{\gamma}}\asin{\left(\dfrac{2\alpha \sin\gamma}{\sqrt{1+\alpha^4-2\alpha^2\cos2\gamma}}\right)},\\ k_2:=\dfrac{1}{2\alpha\sin{\gamma}}\log{\dfrac{\sqrt{1+\alpha^4-2\alpha^2\cos2\gamma}+2\alpha \sin\gamma}{1-\alpha^2}}
    \intertext{and}
    k:=\dfrac{1}{2\alpha\sin{\gamma}}\atan{\left(\dfrac{\alpha^2\sin{2\gamma}}{\alpha^2\cos{2\gamma}-1}\right)}.
 \end{gather*}
Now we shall discuss the image domain of $\psi_{A,B}$ for different cases of $\alpha$: \\ For  $0<\alpha<1,$ $\psi_{A,B}(\mathbb{D})=\{u+iv\in\mathbb{C}:(u,v)\in\Omega_1\}$ where

\begin{equation}\label{ellipse}
   \Omega_1:=\begin{cases} \dfrac{u^2}{h_1^2}+\dfrac{v^2}{h_2^2}<1, &\text{when}\; A=-B=\alpha \\
   \dfrac{(u-k)^2}{k_1^2}+\dfrac{v^2}{k_2^2}<1,&\text{when}\; A=\alpha e^{i \gamma}, B=\alpha e^{-i \gamma}\; \text{for}\;0<\gamma\leq\pi/2.\end{cases}
\end{equation}

 Moreover when $\alpha=1$, the major axis $h_1$ and $k_2$ respectively, tends to infinity, hence $\psi_{A,B}(\mathbb{D})=\{u+iv\in\mathbb{C}:(u,v)\in\Omega_2\}$ where $\Omega_2$ is the following  strip domain
 \begin{equation}\label{strip}
     \Omega_2:=\begin{cases}  -\dfrac{\pi}{4}<v<\dfrac{\pi}{4},& \text{when}\; A=-B=1\\ \dfrac{\gamma-\pi}{2\sin{\gamma}}<u<\dfrac{\gamma}{2\sin{\gamma}},& \text{when}\; A=e^{i \gamma}, B=e^{-i \gamma}\; \text{for} \;0<\gamma\leq\pi/2.\end{cases}
 \end{equation}
\begin{remark}\label{rmk1}
\begin{itemize}\item[(i)] $\max_{|z|=r}\RE\psi_{A,B}(z)=\psi_{A,B}(r)$ and $\min_{|z|=r}\RE\psi_{A,B}(z)=\psi_{A,B}(-r).$
\item[(ii)]  If $\alpha<1$, then  $p\prec\psi_{A,B}(z)$ if and only if $p(\mathbb{D})\subseteq\Omega_{1}$ and if $\alpha=1$, then  $p\prec\psi_{A,B}(z)$ if and only if $p(\mathbb{D})\subseteq\Omega_{2},$ where $\Omega_{1}$ and $\Omega_{2}$ are as defined in (\ref{ellipse}) and (\ref{strip}) respectively.
\item[(iii)] $1+\psi_{A,B}\in\widetilde{\mathcal{M}}_A.$\end{itemize}
\end{remark}

Now in the following lemma, we aim to find the radius of the smallest (largest) disk  centered at the point (1, 0) that  can subscribe (inscribe)
the domain $1+\psi_{A,B}(\mathbb{D})$.

\begin{lemma}\label{inc}
Let $\alpha\in(0,1)$ and $\gamma\in(0,\pi/2]$, then we have \begin{itemize}
\item[(i)] for $A=-B=\alpha$, \begin{equation*}
\left\{w\in\mathbb{C}:|w-1|<h_2\right\}\subset1+\psi_{A,B}(\mathbb{D})\subset\left\{w\in\mathbb{C}:|w-1|<h_1\right\}.
\end{equation*}
\item[(ii)] for $A=\alpha e^{i \gamma}$ and $B=\alpha e^{-i\gamma}$,
\begin{equation*}
\left\{w\in\mathbb{C}:|w-1|<k_1+k\right\}\subset1+\psi_{A,B}(\mathbb{D})\subset\left\{w\in\mathbb{C}:|w-1|<k_2\right\}.
\end{equation*}
\end{itemize}Further for $\alpha=1$,  we have \begin{itemize}
\item[(iii)] for $A=-B=1$, \begin{equation*}
\left\{w\in\mathbb{C}:|w-1|<\dfrac{\pi}{4}\right\}\subset1+\psi_{A,B}(\mathbb{D}).
\end{equation*}
\item[(iv)] for $A= e^{i \gamma}$ and $B= e^{-i\gamma}$,
\begin{equation*}
\left\{w\in\mathbb{C}:|w-1|<\dfrac{\gamma}{2\sin{\gamma}}\right\}\subset1+\psi_{A,B}(\mathbb{D}).
\end{equation*}
\end{itemize}
\end{lemma}
\begin{proof}
Let $w=1+\psi_{A,B}(z)$, then we have $$|w-1|=|\psi_{A,B}(z)|=\dfrac{1}{|A-B|}\bigg|\log\left(\dfrac{1+Az}{1+Bz}\right)\bigg|.$$
It can be easily seen that for $\alpha\in(0,1)$, $A=-B=\alpha$ $$\min_{|z|=1}|\psi_{A,B}(z)|=h_2\;\;\text{and}\;\;\max_{|z|=1}|\psi_{A,B}(z)|=h_1$$
and for  $A=\alpha e^{i \gamma}$ and $B=\alpha e^{-i\gamma}$ $$\min_{|z|=1}|\psi_{A,B}(z)|=k_1+k\;\;\text{and}\;\;\max_{|z|=1}|\psi_{A,B}(z)|=k_2.$$
Further for $\alpha=1$,  $h_1$ and $k_2$ tends to $\infty$, $h_2=\pi/4$ and $k_1+k=\gamma/(2\sin\gamma)$. This completes the proof.
\end{proof}

\begin{remark}In general, if we consider a disc $\mathcal{D}(a,r):=\{w\in\mathbb{C};|w-a|<r\}$ where $a\in\mathbb{R}$ and  $1\in\mathcal{D}(a,r),$ then we observe that $\mathcal{D}(a,r)\subset1+\psi_{A,B}(\mathbb{D})$ if and only if
\begin{equation}
\label{disc1}\mathcal{D}(a,r)\subset\begin{cases}\mathcal{D}(1,h_{2}),& \text{for}\;A=-B=\alpha,\\
\mathcal{D}(1+k,k_1),&\text{for}\; A=\alpha e^{i\gamma},\; B=\alpha e^{-i\gamma}.
\end{cases}\end{equation}
\end{remark}
Thus we arrive at the following:
\begin{lemma}\label{nonempty}
Let  $-1<D<C\leq1$ and $p(z):=(1+Cz)/(1+Dz)$. Then $p(z)\prec\mathcal{L}[A,B]$ if and only if
\begin{itemize}
    \item[(i)] for $A=-B=\alpha$
    \begin{equation}\label{con1}C\leq\begin{cases}h_{2}+(1-h_2)D,& \text{when}\;(1-CD)/(1-D^2)\leq1,\\
h_{2}+(1+h_2)D,& \text{when}\;(1-CD)/(1-D^2)\geq1.
\end{cases}\end{equation}
\item[(ii)] for $A=\alpha e^{i\gamma}$ and $B=\alpha e^{-i\gamma}$,
\begin{equation}\label{con2}C\leq\begin{cases}k_1-k+(1-k_1+k)D,& \text{when}\;(1-CD)/(1-D^2)\leq1+k,\\
k_1+k+(1+k_1+k)D,& \text{when}\;(1-CD)/(1-D^2)\geq1+k.
\end{cases}\end{equation}
\end{itemize}
\end{lemma}
\begin{proof}
Since $p(z)=(1+Cz)/(1+Dz)$ maps $\mathbb{D}$ onto  $\mathcal{D}(a,r)$, where $$a=\dfrac{1-CD}{1-D^2}\quad \text{and} \quad r=\dfrac{C-D}{1-D^2},$$
 the result follows at once from  (\ref{disc1}).
\end{proof}
From Lemma \ref{nonempty}, we deduce the following corollary, which ensures that $\mathcal{F}[A,B]$ is nonempty.
\begin{corollary}
Let  $-1<D<C\leq1$ and $f\in\mathcal{A}$ be such that $zf'(z)/f(z)=(1+Cz)/(1+Dz)$, then $f\in\mathcal{F}[A,B]$ if and only if, the condition (\ref{con1}) or (\ref{con2}) holds.
\end{corollary}

Now from (\ref{sq[a,b]}), we have  $f\in\mathcal{F}[A,B]$ if and only if there exists an analytic function   $p(z)\prec\psi_{A,B}(z)$ such that
\begin{equation}\label{ex}
    f(z)=z\exp\left(\int_{0}^{z}\dfrac{p(t)}{t}dt\right).
\end{equation}
If we take $p(z) = \psi_{A,B}(z),$ then we obtain from (\ref{ex})
\begin{equation}\label{extremal}
    f_{A,B}(z):=z\exp\dfrac{Li_{2}(-B z)-Li_{2}(-A z)}{A-B},
\end{equation}
where $Li_2(x) =\sum_{n=1}^{\infty}x^n/n^2$, denotes the Spence’s (or dilogarithm) function. The function $f_{A,B}(z)$ is nonunivalent but have the extremal properties for many  problems for the class $\mathcal{F}[A,B]$ (see Fig:\ref{fq}).  Kumar and Gangania \cite{kamal} studied the class $\mathcal{F}(\psi)$ as defined in (\ref{kamal}) and obtained  growth and covering theorem for the case when $1+\psi(z)\nprec(1+z)/(1-z)$. As a consequence of the same and Remark \ref{rmk1}, we obtain the following result:

\begin{lemma}
Let $f\in\mathcal{F}[A,B]$ and $f_{A,B}$ be defined as in (\ref{extremal}), then for $|z|=r$\begin{itemize}
    \item[(i)] Growth Theorem: $-f_{A,B}(-r)\leq|f(z)|\leq f_{A,B}(r).$
    \item[(ii)]Covering Theorem: Either $f$ is a rotation of $f_{A,B}$ or $\{w\in\mathbb{C}:|w|\leq -f_{A,B}(-1)\}\subset f(\mathbb{D})$.
\end{itemize}
\end{lemma}
\begin{figure}[ht]
    \centering
    \subfloat[\centering]{\includegraphics[width=40mm, height=40mm]{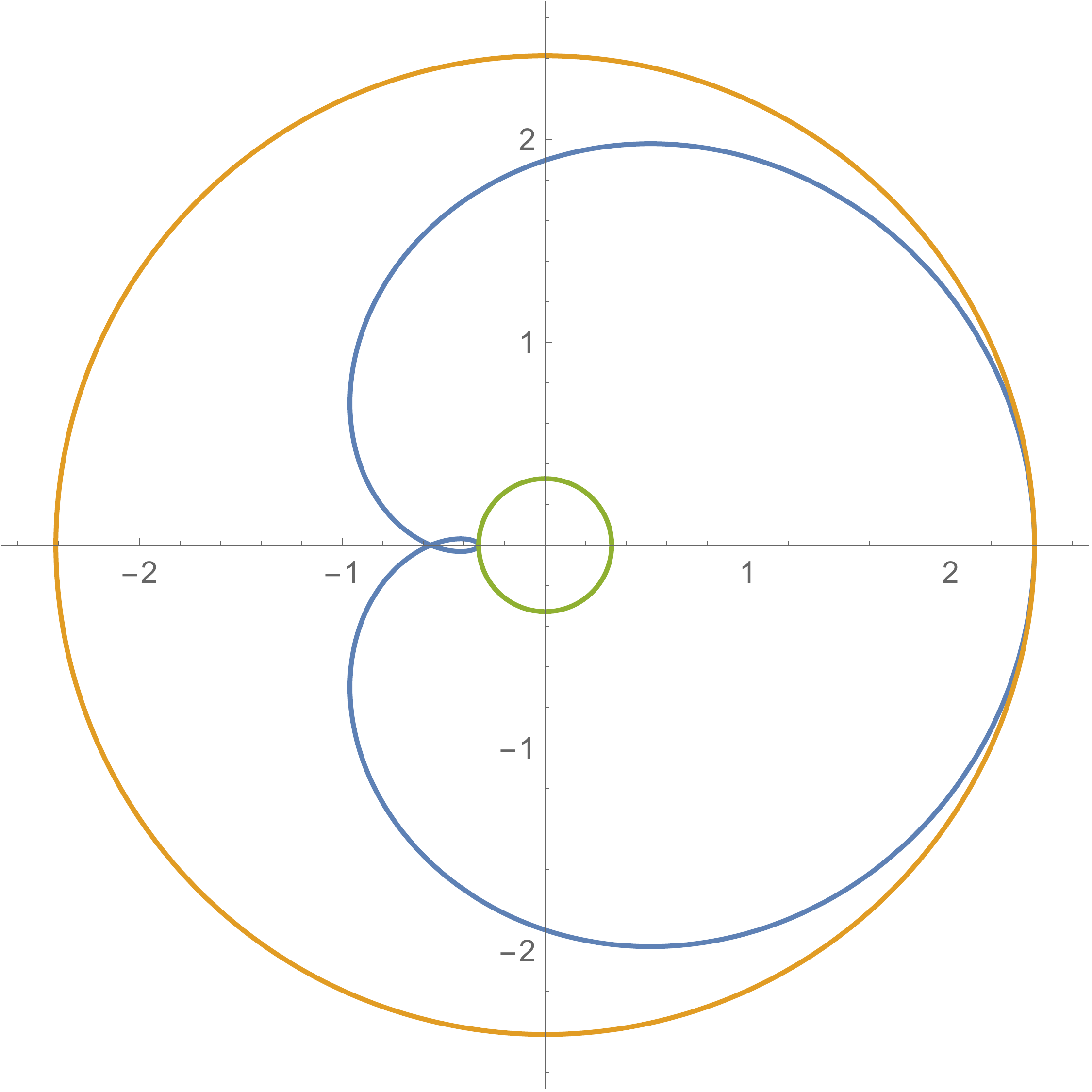}}
\qquad
   \subfloat[ \centering]{\includegraphics[width=30mm, height=40mm]{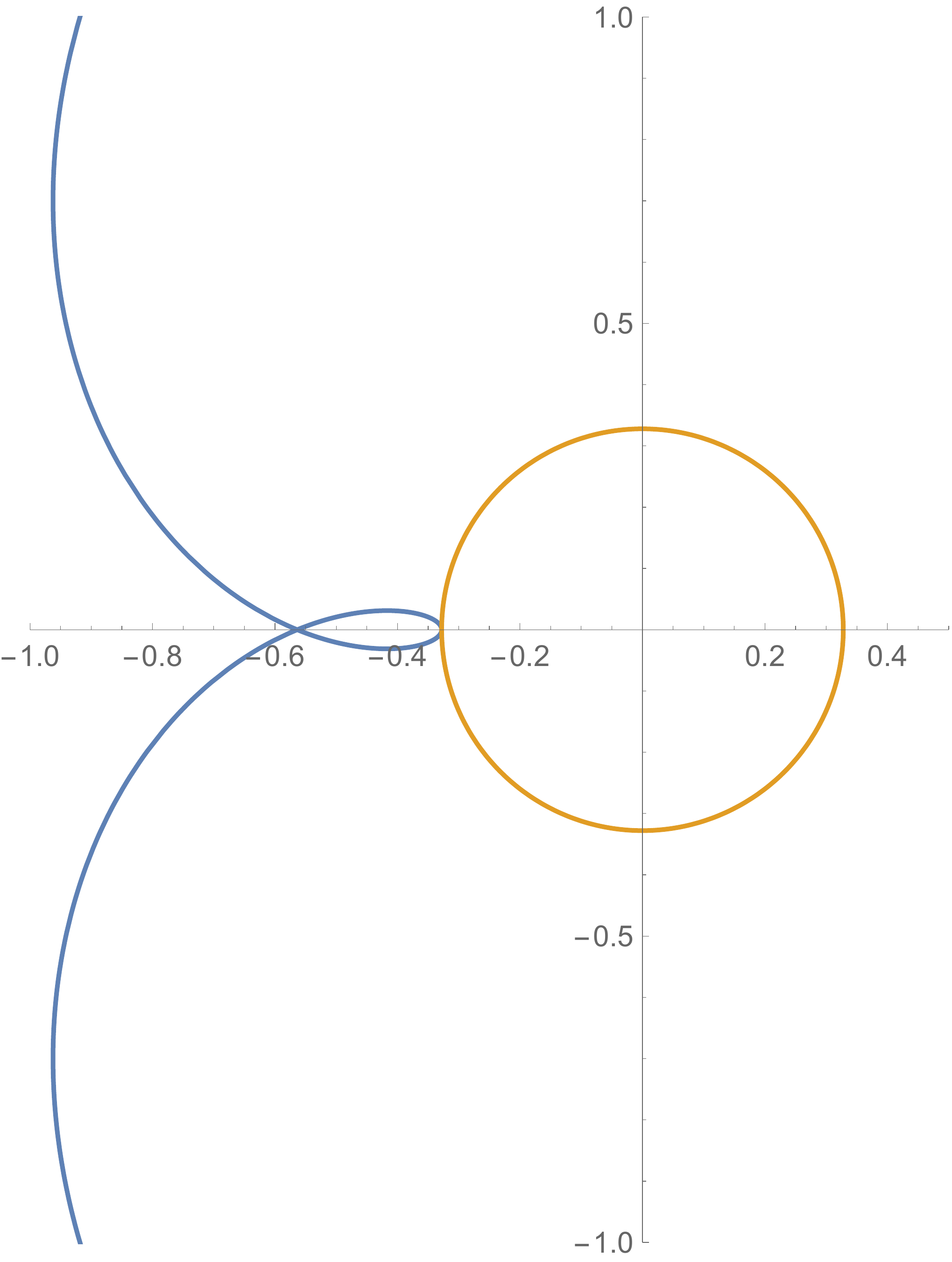}}
    \caption{(A) The images of   $\partial{\mathbb{D}_{f_{A,B}(1)}}$, $f_{A,B}(\partial\mathbb{D})$ and $\partial{\mathbb{D}_{f_{A,B}(-1)}}$, for $A=0.5e^{i\pi/3}$ and $B=0.5e^{-i\pi/3}$ and (B) zoomed image of cusp.}
    \label{fq}
\end{figure}

For $z=r e^{i\theta},$ where $\theta$ is fixed but arbitrary, as a consequence of growth theorem and $\psi_{A,B}(-r)\leq\RE\psi_{A,B}(r e^{i \theta})\leq\psi_{A,B}(r)$,  we obtain
\begin{gather*}
    \log\dfrac{f(z)}{z}=\int_{0}^{r}\dfrac{p(t e^{i \theta})}{t}dt,
    \intertext{where $p(z):=\psi_{A,B}(w(z))$ and $w$ is a Schwarz functon. Further}
    \dfrac{f(z)}{z}=\exp{\int_{0}^{r}\dfrac{p(t e^{i \theta})}{t}dt}
    =\exp\left(\int_{0}^{r}\RE\dfrac{p(t e^{i \theta})}{t}dt+i\int_{0}^{r}\IM\dfrac{p(t e^{i \theta})}{t}dt\right)\\
    \bigg|\dfrac{f(z)}{z}\bigg|\leq\exp\int_{0}^{r}\dfrac{\RE\psi_{A,B}(t e^{i \theta})}{t}dt\\
    \exp\int_{0}^{r}\frac{\psi_{A,B}(-t)}{t}dt\leq  \bigg|\dfrac{f(z)}{z}\bigg|\leq  \exp\int_{0}^{r}\frac{\psi_{A,B}(t)}{t}dt.
\end{gather*}
 Here $L(f, r):=\int_{0}^{2\pi}|zf'(z)|d\theta$ is the length of the boundary curve
$f(|z|=r)$. Now we obtain the following result:
\begin{corollary}
Let $f\in\mathcal{F}[A,B]$ and $M(r)=\exp\int_{0}^{r}\frac{\psi_{A,B}(t)}{t}dt$, then for $|z|=r$, we have
\begin{itemize}
\item[(i)] $M(-r)\leq\bigg|\dfrac{f(z)}{z}\bigg|\leq M(r)$.
\item[(ii)] $(1+\max_{|z|\leq r}|\psi_{A,B}(z)|)M(-r)\leq|f'(z)|\leq(1+\max_{|z|\leq r}|\psi_{A,B}(z)|)M(r)$.
\item[(iii)] $2\pi r(1+\max_{|z|\leq r}|\psi_{A,B}(z)|)M(-r)\leq L(f,r)\leq2\pi r(1+\max_{|z|\leq r}|\psi_{A,B}(z)|)M(r)$.
\item[(iv)] $f(z)/z\prec f_{A,B}(z)/z.$
\end{itemize}
\end{corollary}

\section{Radius Estimates}
In this section, we obtain some radii estimates.
From Theorem \ref{bounds} we conclude that $f\in\mathcal{S}[A,B]$ is starlike of order $ 1+\dfrac{1}{|A-B|}\left(\varphi-\vartheta\right)<0$, hence $f$ may not be univalent in $\mathbb{D}$, thus $\mathcal{F}[A,B]\not\subseteq\mathcal{S}^*.$  Therefore, in the following result, we find the radius of starlikeness of order $\delta$, where $\delta\in[0,1)$ for functions in the class $\mathcal{F}[A,B]$.

\begin{theorem}\label{rad}
Let $\alpha\in(0,1],$ $\gamma\in(0,\gamma_0)$, where $\gamma_0\simeq1.2461\hdots$ and $\delta\in[0,1)$ be given numbers. If $f\in\mathcal{F}[A,B]$, then $f$ is starlike of order $\delta$ in the disc $|z|<r(\delta)$, where
\begin{itemize}
    \item[(i)] For $A=-B=\alpha$,
  \begin{equation}\label{r1}  r(\delta)=\dfrac{\exp(2\alpha(1-\delta))-1}{\alpha(\exp(2\alpha(1-\delta))+1)}.\end{equation}
    \item[(ii)] For $A=\alpha e^{i \gamma}$ and $B=\alpha e^{-i\gamma}$, $r(\delta)$ is the smallest positive root of

\begin{equation}\label{r2}
   \tan(2\alpha\sin{\gamma}(\delta-1))= \dfrac{\alpha^4r^4\sin{2\gamma}+2\alpha^3r^3\sin{\gamma}\cos{2\gamma}-\alpha^2r^2\sin{2\gamma}-2\alpha r\sin{\gamma}}{\alpha^4r^4\cos{2\gamma}-2\alpha^3r^3\sin{\gamma}\sin{2\gamma}-\alpha^2r^2\cos{2\gamma}-\alpha^2r^2+1}.
\end{equation}
\end{itemize}
The result is sharp.
\end{theorem}

\begin{proof}
Let $f\in\mathcal{F}[A,B]$, then using Theorem \ref{bounds} for $|z|<r$, we have \\
$\textbf{Case 1.}$ when $A=-B=\alpha$,
\begin{equation*}
    \RE \dfrac{zf'(z)}{f(z)}>1-\dfrac{1}{2\alpha}\log\left(\dfrac{1+\alpha r}{1-\alpha r}\right)=:t_1(r),
\end{equation*} Now as $t_1'(r)<0$ for all choices of $\alpha$, thus $t_1(r)$ is strictly decreasing function  from $1$ to $1-1/(2\alpha)\log((1+\alpha)/(1-\alpha))<0.$ Hence the  root $r(\delta)$, given in (\ref{r1}), of the equation  $t_1(r)=\delta$, is the radius of starlikeness of order $\delta$ of $\mathcal{F}[A,B]$. \\
$\textbf{Case 2.}$ when $A=\alpha e^{i \gamma}$ and $B=\alpha e^{-i\gamma}$
\begin{equation*}
    \RE \dfrac{zf'(z)}{f(z)}>1- \dfrac{1}{2\alpha\sin{\gamma}}(\eta-\tau)=:t_2(r),
\end{equation*}where $\eta$ and $\tau$ are given in Theorem \ref{bounds}.  Since $t_2(r)=\delta$ can be simplified to the equation  (\ref{r2})  by a simple computation, thus it is enough to find the smallest positive root of  (\ref{r2}).
Further
$$t_2'(r)=-\dfrac{1-\alpha^4r^4+2\alpha r \cos{\gamma}(1-\alpha^2r^2)}{(1-\alpha^2r^2)\sqrt{1+\alpha^4r^4-2\alpha^2r^2 \cos{2\gamma}}}<0$$ for all choices of $\alpha$ and $\gamma$, $t_2(r)$ is a strictly decreasing function  from $1$ to $g(\alpha,\gamma)$, where
 \begin{equation*}
    g(\alpha,\gamma):=1+\dfrac{1}{2\alpha\sin{\gamma}}\left(\atan\left(\dfrac{\alpha^2\sin{2\gamma}}{\alpha^2\cos{2\gamma}-1}\right)-\asin\left(\dfrac{2\alpha \sin\gamma}{\sqrt{1+\alpha^4-2\alpha^2\cos2\gamma}}\right)\right).
\end{equation*}\begin{figure}[ht]
    \centering
    \includegraphics[width=80mm, height=50mm]{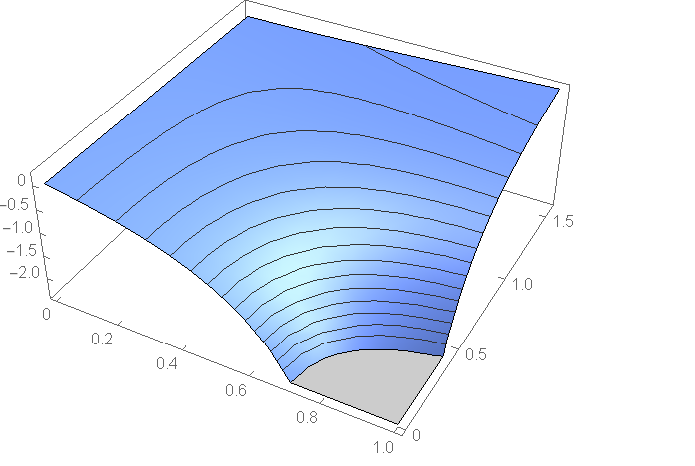}
    \caption{Image of $g(\alpha,\gamma)$ on $R:=(0,1)\times(0,\pi/2)$.}
    \label{h(1)}
\end{figure}Note that $g(\alpha,\gamma)$ is real valued, non-constant analytic function on a bounded domain $R:=(0,1)\times(0,\pi/2)$ and  is bounded above on $R$. Thus by Maximum Modulus Theorem,  $g(\alpha,\gamma)$ attains its maximum value on the boundary of $R$ and the four boundaries are  enlisted below:  \begin{itemize}
    \item[1.] $g(0,\gamma)=0$.
    \item[2.] $g(1,\gamma)=1+(\gamma-\pi)/(2\sin\gamma)$. Now as $$g'(1,\gamma)=(1+(\pi-\gamma)\cot\gamma)/(2\sin\gamma)>0,$$
    therefore $g(1,\gamma)$ is strictly increasing from $-\infty$ to $1-\pi/4\simeq0.2146\hdots$ and $\gamma_0$ is the zero of it.
    \item[3.] $g(\alpha,0)=1-1/(1-\alpha)$. As $g'(\alpha,0)=\log(1-\alpha)<0$, therefore $g(\alpha,0)$ is strictly decreasing from $0$ to $-\infty$.
    \item[4.] $g(\alpha,\pi/2)=1-\left(\asin\left(2\alpha/(1+\alpha^2)\right)\right)/(2\alpha)$. As 
    $$g'(\alpha,\pi/2)=\\ \left(\left(2\alpha/(1+\alpha^2)\right)+ \asin\left(2\alpha/(1+\alpha^2)\right)\right)/(2\alpha^2)>0,$$ 
    therefore $g(\alpha,\pi/2)$ is also strictly increasing from $0$ to $1-\pi/4\simeq0.2146$.
\end{itemize}
From Fig. \ref{h(1)} and all four boundaries, we observe that $g(\alpha,\gamma)<0$ for all choices of  $\alpha\in(0,1],$ $\gamma\in(0,\gamma_0)$, where $\gamma_0$ is the root of the equation
$g(1,\gamma)=0$.
    Hence $r(\delta)$ is the radius of starlikeness of order $\delta$ of $\mathcal{F}[A,B]$.
\end{proof}

\begin{theorem}\label{rad1}
Let $\alpha\in(0,1]$  and $\delta\in[0,1)$. If $f\in\mathcal{F}[A,B]$, for the case when $A=\alpha e^{i \gamma}$ and $B=\alpha e^{-i\gamma}$, then there is an unique $\gamma'\in(\gamma_0,\pi/2)$ such that
\begin{itemize} \item[(i)]  when $\gamma\in(\gamma_0,\gamma']$, $f$ is starlike of order $\delta$ in $\mathbb{D}_{r(\delta)},$ where $r(\delta)$ the root of the equation $(\ref{r2}).$
\item[(ii)] when $\gamma\in(\gamma',\pi/2)$, $f$ is starlike of order $\delta\in[0,g(\alpha,\pi/2))$ in $\mathbb{D}$ and starlike of order $\delta\in[g(\alpha,\pi/2),1)$ in $\mathbb{D}_{r(\delta)}.$ \end{itemize}
\end{theorem}
\begin{proof}
By observing the Fig. \ref{h(1)} and the nature of all boundary curves discussed in Theorem \ref{rad},  we arrive at  $g(\alpha,\gamma_0)<0$  and $g(\alpha,\pi/2)>0$ for all $\alpha\in(0,1)$. Thus by IVP for any choice of $\alpha\in(0,1)$, there is an unique $\gamma'\in(\gamma_0,\pi/2)$ such that $g(\alpha,\gamma')=0$, $g(\alpha,\gamma)<0$ for all $\gamma\in(\gamma_0,\gamma')$ and $g(\alpha,\gamma)>0$ for all $\gamma\in(\gamma',\pi/2).$ Hence the result follows.
\end{proof}
Taking $\delta=0$ in Theorems \ref{rad} and \ref{rad1}, we obtain the following result:
\begin{corollary}\label{cr}
Let $\alpha\in(0,1]$  and $\gamma\in(0,\pi/2]$. If $f\in\mathcal{F}[A,B]$ then
\begin{itemize}
    \item[(i)] For $A=-B=\alpha$, $f$ is starlike univalent in the disc $|z|<r_0$, where
  \begin{equation*}  r_0=\dfrac{\exp(2\alpha)-1}{\alpha(\exp(2\alpha)+1)}.\end{equation*}
    \item[(ii)] For $A=\alpha e^{i \gamma}$ and $B=\alpha e^{-i\gamma}$, $f\in\mathcal{S}^*$, whenever  $\gamma\in(\gamma',\pi/2)$,  and $f$ is starlike univalent in the disc $|z|<r_0$, whenever $\gamma\in(0,\gamma_0)\cup(\gamma_0,\gamma']$, where $\gamma_0$, $\gamma'$ are given in Theorems \ref{rad} and \ref{rad1}, respectively and $r_0$ is the smallest positive root of

\begin{equation*}
   \tan(2\alpha\sin{\gamma})+ \dfrac{\alpha^4r^4\sin{2\gamma}+2\alpha^3r^3\sin{\gamma}\cos{2\gamma}-\alpha^2r^2\sin{2\gamma}-2\alpha r\sin{\gamma}}{\alpha^4r^4\cos{2\gamma}-2\alpha^3r^3\sin{\gamma}\sin{2\gamma}-\alpha^2r^2\cos{2\gamma}-\alpha^2r^2+1}=0.
\end{equation*}\end{itemize}The result is sharp.
\end{corollary}

\begin{corollary}
Let $f\in\mathcal{F}[A,B]$. Then for the disc $|z|< r\leq r_0$, where $r_0$ is given in Corollary \ref{cr}, we have
\begin{itemize}
\item[(i)] For $A=-B=\alpha$,
\begin{equation*}
\bigg|\arg\biggl(\dfrac{zf'(z)}{f(z)}\biggr)\bigg|<\atan\left(\dfrac{\asin\left(\frac{2\alpha r}{1+\alpha^2 r^2}\right)}{2\alpha-\log\left(\frac{1+\alpha r}{1-\alpha r}\right)}\right)
\end{equation*}
\item[(ii)] For $A=\alpha e^{i \gamma}$ and $B=\alpha e^{-i\gamma}$,
\begin{equation*}
\bigg|\arg\biggl(\dfrac{zf'(z)}{f(z)}\biggr)\bigg|<\atan\biggl(\dfrac{\log T_2}{2\alpha\sin{\gamma}-\eta+\tau}\biggr),
\end{equation*}where $T_2$, $\eta$ and $\tau$ are given in Theorem \ref{bounds}.
\end{itemize}
\end{corollary}
\begin{corollary}
Let $f\in\mathcal{F}[A,B]$. Then for the disc $|z|< r_s\leq r_0$,  where $r_0$ is given in Corollary \ref{cr}, $f\in\mathcal{SS}^*(\beta)$, where $r_s\in(0, r_0]$ is the smallest positive root of the following equation:
\begin{itemize}
\item[(i)] For $A=-B=\alpha$,
\begin{equation*}
\asin\left(\dfrac{2\alpha r}{1+\alpha^2 r^2}\right)=\tan\left(\dfrac{\beta\pi}{2}\right)\left(2\alpha-\log\left(\frac{1+\alpha r}{1-\alpha r}\right)\right)
\end{equation*}
\item[(ii)] For $A=\alpha e^{i \gamma}$ and $B=\alpha e^{-i\gamma}$,
\begin{equation*}
\dfrac{\log T_2}{2\alpha\sin{\gamma}-\eta+\tau}=\tan\left(\dfrac{\beta\pi}{2}\right),
\end{equation*}where $T_2$, $\eta$ and $\tau$ are given in Theorem \ref{bounds}.
\end{itemize}
\end{corollary}

We now discuss the radius estimates for  the classes $\mathcal{BS}^*(\alpha)$  and $\mathcal{S}^*_{cs},$ which are not contained in $\mathcal{S}^*$. Further we derive radii for functions in $\mathcal{BS}^*(\alpha)$ and $\mathcal{S}^*_{cs}$ to be in $\mathcal{F}[A,B].$
\begin{theorem}
Let $\alpha\in(0,1)$, $r_{0}=(-1+\sqrt{1+4\alpha h_1^2})/(2\alpha h_1)$ and $\alpha_0\in(0,1)$ is the smallest solution of \begin{equation*}
 1-\dfrac{r_0^2 \cos^2{\theta}(1-\alpha r_0^2)^2}{(1+\alpha^2 r_0^4-2\alpha r_0^2\cos{2\theta})^2h_1^2}-\dfrac{r_0^2 \sin^2{\theta}(1+\alpha r_0^2)^2}{(1+\alpha^2 r_0^4-2\alpha r_0^2\cos{2\theta})^2h_2^2}=0,
 \end{equation*}where $\theta\in(0,\pi/2)$. If $f\in\mathcal{BS}^*(\alpha)$ then for the disc $|z|\leq r_b<1$, $f\in\mathcal{F}[A,B]$, where
\begin{equation*}r_b=\begin{cases}r_0,&\text{for}\; A=-B=\alpha\leq\alpha_0\\
r_1:=\dfrac{-1+\sqrt{1+4\alpha h_2^2}}{2\alpha h_2},&\text{for}\; A=-B=\alpha>\alpha_0\\
r_2:=\dfrac{-1+\sqrt{1+4\alpha (k_1+k)^2}}{2\alpha (k_1+k)},&\text{for}\; A=\alpha e^{i \gamma}\; \text{and}\; B=\alpha e^{-i\gamma}.
\end{cases}
\end{equation*} Moreover the radii $r_0$ and $r_2$ are sharp.
\end{theorem}
\begin{proof}
Let $f\in\mathcal{BS}^*(\alpha)$. Since for $A=-B=\alpha$, $\max_{|z|=1}\RE \psi_{A,B}(z)=h_1$. Thus to find such $r<1$ for which the image of $zf'(z)/f(z)-1$ under the disc $|z|<r$ lies inside $\psi_{A,B}(\mathbb{D})$, it is necessary that \begin{equation}\label{nec1}\max_{|z|=r<1}\RE\left(\dfrac{z}{1-\alpha z^2}\right)=\dfrac{r}{1-\alpha r^2}\leq h_1\end{equation} must hold. Clearly for $|z|\leq r_0$, the inequality (\ref{nec1}) holds. Now to see that for  $\alpha\leq \alpha_0$, radius $r_0$ is also sufficient for $zf'(z)/f(z)-1\prec z/(1-\alpha z^2)\in\psi_{A,B}(\mathbb{D})$ in the disc $|z|\leq r_0$. For $\zeta=r e^{i\theta}$ $(\theta\in[0,2\pi))$, we have $$B_r(\theta):=\dfrac{\zeta}{1-\alpha\zeta^2}=\dfrac{r \cos{\theta}(1-\alpha r^2)}{1+\alpha^2 r^4-2\alpha r^2\cos{2\theta}}+i \dfrac{r \sin{\theta}(1+\alpha r^2)}{1+\alpha^2 r^4-2\alpha r^2\cos{2\theta}}.$$ Since $\RE B_r(\theta)=\RE B_r(-\theta)$, $\RE B_r(\theta)=-\RE B_r(\pi-\theta)$ and $\IM B_r(\theta)=\IM B_r(\pi-\theta)$, therefore the curve $B_r(\theta)$ is symmetric about real and imaginary axis thus it is sufficient  to consider for $\theta\in[0,\pi/2].$ Now for $r=r_0,$ the square of the distance from the origin to the points of $B_{r_0}(\theta)$ is given by \begin{equation*}
Dist(0;B_{r_0}(\theta)):=\dfrac{r_0^2}{1+\alpha^2 r_0^4-2\alpha r_0^2\cos{2\theta}}.
\end{equation*} Since $Dist(0;B_{r_{0}}(\theta))'<0$, thus $Dist(0;B_{r_0}(\theta))$ is a decreasing function of $\theta.$ Hence the farthest  point of $B_{r_0}(\theta)$ from origin is $(r_0/(1-\alpha r_0^2),0)$, which lies on the boundary of $\Omega_1$. Now for $A=-B=\alpha>\alpha_0$ and $A=\alpha e^{i \gamma}$, $B=\alpha e^{-i\gamma}$ with $|z|=r$, we have $$\bigg|\dfrac{zf'(z)}{f(z)}-1\bigg|\leq\dfrac{r}{1-\alpha r^2}.$$ Therefore, by Lemma \ref{inc}, we see that $\mathcal{F}[A,B]-$radius for the class $\mathcal{BS}^*(\alpha)$ are the smallest positive roots $r_1$ and $r_2$ of the equations $r/(1-\alpha r^2)= h_2$ and $r/(1-\alpha r^2)= k_1+k,$ respectively.
\end{proof}

\begin{theorem}
Let $\alpha\in(0,1)$ and $\alpha_0\in(0,1)$ be the smallest solution of 
\begin{equation*}
\dfrac{r_0^2 ((\alpha-1)r_0+\cos{\theta}(1-\alpha r_0^2))^2}{h_1^2}+\dfrac{r_0^2 (1+\alpha r_0^2)^2\sin^2{\theta}}{h_2^2}=N,
 \end{equation*}
 where $\theta\in(0,\pi/2)$ and
 $$N=(1+r_0^2-2r_0\cos\theta)^2(1+\alpha^2 r_0^2+2\alpha r_0\cos{\theta})^2$$
 and
 $$r_{0}=(-(1+(1-\alpha)h_1)+\sqrt{(1+(1-\alpha)h_1)^2+4\alpha h_1^2})/(2\alpha h_1).$$
 If $f\in\mathcal{S}^*_{cs}(\alpha)$ then for the disc $|z|\leq r_{cs}<1$, $f\in\mathcal{F}[A,B]$, where
\begin{equation*}\resizebox{.95\hsize}{!}{$r_{cs}=\begin{cases}r_0,&\text{for}\; A=-B=\alpha\leq\alpha_0\\
r_1:=\dfrac{-(1+(1-\alpha)h_2)+\sqrt{(1+(1-\alpha)h_2)^2+4\alpha h_2^2}}{2\alpha h_2},&\text{for}\; A=-B=\alpha>\alpha_0\\
r_2:=\dfrac{-(1+(1-\alpha)(k_1+k))+\sqrt{(1+(1-\alpha)(k_1+k))^2+4\alpha (k_1+k)^2}}{2\alpha (k_1+k)},&\text{for}\; A=\alpha e^{i \gamma}, B=\alpha e^{-i\gamma}.
\end{cases}$}\end{equation*} Moreover the radii $r_0$ and $r_2$ are sharp.
\end{theorem}
\begin{proof}
Let $f\in\mathcal{S}^*_{cs}(\alpha)$. Since for $A=-B=\alpha$, $\max_{|z|=1}\RE \psi_{A,B}(z)=h_1$. Thus to find such $r<1$ for which the image of $zf'(z)/f(z)-1$ under the disc $|z|<r$ lies inside $\psi_{A,B}(\mathbb{D})$, it is necessary that \begin{equation}\label{nec}\max_{|z|=r<1}\RE\left(\dfrac{z}{(1-z)(1+\alpha z)}\right)=\dfrac{r}{(1-r)(1+\alpha r)}\leq h_1\end{equation} must hold. Clearly for $|z|\leq r_0$, the equation (\ref{nec}) holds. Now to see that for  $\alpha\leq \alpha_0$ radius $r_0$ is also sufficient for $zf'(z)/f(z)-1\prec z/((1-z)(1+\alpha z))\in\psi_{A,B}(\mathbb{D})$ in the disc $|z|\leq r_0$. For $\zeta=r e^{i\theta}$ $(\theta\in[0,2\pi))$, we have

\begin{align*}
CS_r(\theta):=\dfrac{\zeta}{(1-\zeta)(1+\alpha\zeta)}=&\dfrac{r ((\alpha-1)r+\cos{\theta}(1-\alpha r_0^2))}{(1+r^2-2r\cos\theta)(1+\alpha^2 r^2+2\alpha r\cos{\theta})}+\\
&i\dfrac{r (1+\alpha r_0^2)\sin{\theta}}{(1+r^2-2r\cos\theta)(1+\alpha^2 r^2+2\alpha r\cos{\theta})}.
\end{align*}
Since $\RE CS_r(\theta)=\RE CS_r(-\theta)$, therefore the curve $CS_r(\theta)$ is symmetric about real axis thus it is sufficient  to consider for $\theta\in[0,\pi].$ Now for $r=r_0,$ the square of the distance from the origin to the points of $CS_{r_0}(\theta)$ is given by \begin{equation*}
Dist(0;CS_{r_0}(\theta)):=\dfrac{r_0^2}{(1+r^2-2r\cos\theta)^2(1+\alpha^2 r^2+2\alpha r\cos{\theta})^2}.
\end{equation*} Since $Dist(0;CS_{r_{0}}(\theta))'=0$ for $\theta=0,\theta_{0}$ and $\pi$ with $Dist(0;CS_{r_{0}}(\theta))'<0$ whenever $\theta\in(0,\theta_0)$ and  $Dist(0;CS_{r_{0}}(\theta))'>0$ whenever $\theta\in(\theta_0,\pi)$.  And  $Dist(0;CS_{r_0}(0))-Dist(0;CS_{r_0}(\pi))>0$, hence the farthest  point of $CS_{r_0}(\theta)$ from origin is equal to $h_{1}$ obtained at $\theta=0$. Now for $A=-B=\alpha>\alpha_0$ and $A=\alpha e^{i \gamma}$, $B=\alpha e^{-i\gamma}$ with $|z|=r$, we have $$\bigg|\dfrac{zf'(z)}{f(z)}-1\bigg|\leq\dfrac{r}{(1-r)(1+\alpha r)}.$$ Therefore, by Lemma \ref{inc}, we see that $\mathcal{F}[A,B]-$radius for the class $\mathcal{S}^*_{cs}(\alpha)$ are the smallest positive roots $r_1$ and $r_2$ of the equations $r/((1-r)(1+\alpha r))= h_2$ and $r/((1-r)(1+\alpha r))= k_1+k,$ respectively.
\end{proof}

\begin{figure}[ht]
    \centering
    \subfloat[\centering]{\includegraphics[width=40mm, height=30mm]{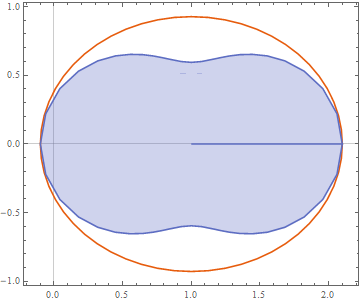}}
\qquad
   \subfloat[ \centering]{\includegraphics[width=40mm, height=30mm]{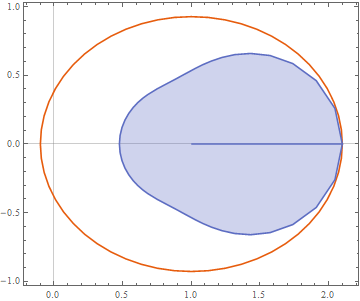}}
    \caption{$f_\alpha(z)=z/(1-\alpha z^2)$ and $g_\alpha(z)=z/((1- z)(1+\alpha z))$, (A) $f_{0.5}(\mathbb{D}_{r_0=0.77})\subset\psi_{0.5,-0.5}(\partial\mathbb{D})$ and 
    (B) $g_{0.5}(\mathbb{D}_{r_0=0.59})\subset\psi_{0.5,-0.5}(\partial\mathbb{D})$.}
    \label{CS1}
\end{figure}

\end{document}